\documentclass[12pt]{amsart}
\usepackage{amsmath, amsthm, amscd, amsfonts, amssymb, graphicx, color}
\usepackage[colorlinks]{hyperref}
\newtheorem{thm}{Theorem}[section]
\newtheorem{cor}[thm]{Corollary}
\newtheorem{lem}[thm]{Lemma}
\newtheorem{prop}[thm]{Proposition}
\theoremstyle{definition}

\theoremstyle{remark}

\numberwithin{equation}{section}
\newtheorem{exa}[thm]{Example}

\newcommand{\h}{\mathcal{H}}

\begin{document}

\title[More sums of Hilbert space frames]
{More sums of Hilbert space frames}%

\author[A. Najati,  M.R. Abdollahpour, E. Osgooei and M.M. Saem ]
{ A. Najati$^1$,  M.R. Abdollahpour$^1$, E. Osgooei$^2$ and M.M. Saem$^1$ }%
\address{
\newline
\indent $^1$Department of Mathematical Sciences
\newline \indent   University of Mohaghegh Ardabili
\newline \indent  Ardabil 56199-11367
\newline \indent Iran}
\email{a.nejati@yahoo.com} \email{mrabdollahpour@yahoo.com}
\email{m.mohammadisaem@yahoo.com}
\address{
\newline
\indent $^2$Faculty of Mathematical Sciences
\newline \indent   University of Tabriz
\newline \indent  Tabriz
\newline \indent Iran}
\email{osgooei@tabrizu.ac.ir}


\begin{abstract}
In this paper we have some new results on sums of Hilbert space
frames and Riesz bases. We also have a correction for some results
in "S. Obeidat et al., Sums of Hilbert space frames, J. Math. Anal.
Appl. 351 (2009) 579--585."
\end{abstract}
\maketitle
\section{ Introduction }
Throughout this paper, $\h$ denotes a separable Hilbert space with
the inner product $\langle.,.\rangle$. Recall that a  sequence
$\{f_{i}\}_{i\in I }\subseteq{\h}$ is a frame for $\h$  if there
exist  $0<A\leqslant B<\infty$ such that
\begin{equation}\label{frame}
A\|f\|^{2}\leqslant \sum_{i\i\in I }|\langle
f,f_{i}\rangle|^{2}\leqslant B\|f\|^{2}
\end{equation}
for all $f\in \h.$ The constants $A$ and $B$ are called a lower and
upper frame bound.
\par
If $\{f_{i}\}_{i\in I }\subseteq{\h}$ is a frame for ${\h}$, the
\textit{frame operator} for $\{f_{i}\}_{i\in I}$ is the bounded
linear operator $S:{\h}\rightarrow{\h}$ given by $Sf=\sum_{i\in
I}\langle f,f_i\rangle f_i.$ Therefore $\langle
Sf,f\rangle=\sum_{i\i\in I }|\langle f,f_{i}\rangle|^{2}$ for all
$f\in \h.$ It follows that $S$ is  positive and invertible. This
provides the frame decomposition
\[
f=\sum_{i\in I}\langle f,S^{-1}f_i\rangle
f_i=\sum_{i\in I}\langle f,f_i\rangle S^{-1}f_i
\]
 for all $f\in{\h}.$

\section{ Main results }
The following  is proved in [3, Proposition 2.1].
\begin{prop}\label{pro1}\cite{sum}
Let $\{f_i\}_{i\in I}$ be a frame for $\h$ with the frame operator
$S$, frame bounds $A\leqslant B$ and let $L:\h\longrightarrow \h$ be
a bounded operator. Then $\{Lf_i\}_{i\in I}$ is a frame for $\h$ if
and only if $L$ is invertible on $\h$. Moreover, in this case the
frame operator for $\{Lf_i\}_{i\in I}$ is $L S L^*$ and the new
frame bounds are $A\|L^{-1}\|^{-2}, B\|L\|^2$.
\end{prop}
In this note, we show that  Proposition \ref{pro1}  is not true in
general. Indeed,  if $\{f_i\}_{i\in I}$ is a frame for Hilbert space
$\h$ and $L:\h \longrightarrow \h$ is a bounded invertible operator,
then $\{Lf_i\}_{i\in I}$ is a frame for $\h$ but the inverse is not
true in general. In the proof of Proposition \ref{pro1}, the authors
proved that $LSL^*$ is invertible. It does not imply that $L$ is
invertible on $\h$. It should be noted that Proposition \ref{pro1}
has been used in Corollaries 2.2 , 2.3 and in the proof of
Proposition 4.1 of \cite{sum}.
\begin{exa}\label{el}
Let $\{e_n\}_{n=1}^{\infty}$ be an orthonormal basis for a Hilbert
space $\h$. Define a shift operator $L$ on $\h$ by $L(e_n)=e_{n-1}$
if $n>1$ and $L(e_1)=0$. It is clear that
$\{L(e_n)\}_{n=1}^{\infty}$ is a frame for $\h$, but $L$ is not
invertible although $LL^*=I$. Moreover,
$\{L^*(e_n)\}_{n=1}^{\infty}$ is not a frame for $\h$.
\end{exa}
We can improve Proposition \ref{pro1} as follows:
\begin{prop}\label{pro}
Let $\{f_i\}_{i\in I}$ be a frame for $\h$ with the frame operator
$S$, frame bounds $A\leqslant B$ and let $L:\h\longrightarrow \h$ be
a bounded operator. Then $\{Lf_i\}_{i\in I}$ is a frame for $\h$ if
and only if $L$ is surjective. Moreover, in this case the frame
operator for $\{Lf_i\}_{i\in I}$ is $L S L^*$ and the new frame
bounds are $A\|L^{\dag}\|^{-2}$ and $B\|L\|^2$, where $L^{\dag}$ is
the pseudo-inverse of $L$.
\end{prop}
\begin{proof}
If $\{Lf_i\}_{i\in I}$ is a frame for $\h$, then its frame operator
$LSL^*$ is invertible. So $L$ is surjective. The converse follows
from Corollary 5.3.2 of \cite{Ole}.
\end{proof}
We also have
\begin{prop}
Let $\{f_i\}_{i\in I}$ be a frame for $\h$ with the frame operator
$S$ and let $L:\h\longrightarrow \h$ be a bounded operator. Then
$\{Lf_i\}_{i\in I}$ and $\{L^*f_i\}_{i\in I}$ are frame for $\h$ if
and only if $L$ is invertible. Moreover, in this case the frame
operators for $\{Lf_i\}_{i\in I}$ and $\{L^*f_i\}_{i\in I}$ are  $L
S L^*$ and $L^* S L$, respectively.
\end{prop}
\begin{proof}
If $\{Lf_i\}_{i\in I}$ and $\{L^*f_i\}_{i\in I}$ are frames for
$\h$, then their frame operators $LSL^*$ and $L^* S L$ are
invertible. So $L$ is invertible. The converse is clear.
\end{proof}
In \cite{sum}, corollary 2.2 can be improved as
below.
\begin{cor}
Let $\{f_i\}_{i\in I}$ be a frame for $\h$ with the frame operator
$S$, frame bounds $A\leqslant B$ and let $L:\h\longrightarrow \h$ be
a bounded operator, then $\{f_i+Lf_i\}_{i\in I}$ is a frame for $\h$
if and only if $I+L$ is surjective. Moreover, in this case the frame
operator for the new frame is $(I+L) S (I+L^*)$ with the frame
bounds  $A\|(I+L)^{\dag}\|^{-2}$ and $B\|I+L\|^2$, where
$(I+L)^{\dag}$ is the pseudo-inverse of $I+L$. In particular, if $L$
is a positive operator (or just $L>-I$), then $\{f_i+Lf_i\}_{i\in
I}$ is a frame for $\h$ with the frame operator $S+SL+SL^*+LSL^*.$
\end{cor}
\begin{cor}
Let $\{f_i\}_{i\in I}$ be a frame for $\h$ and $P:\h\longrightarrow \h$ be
a bounded operator. If $P^2=P,$ then for all $a\neq -1,$ $\{f_i+aPf_i\}_{i\in I}$ is a frame for $\h.$
\end{cor}
\begin{proof}
If $a\neq -1,$ then we have $(I+aP)(I-\frac{a}{a+1}P)=I.$ This implies that $I+aP$ is invertible and so
$\{f_i+aPf_i\}_{i\in I}$ is a frame for $\h.$
\end{proof}
\begin{prop}
Let $\{f_i\}_{i\in I}$ be a sequence in $\h$ such that $\sum_{i\in
I}\langle f,f_i\rangle f_i$ converges for all $f\in\h$. If
$L:\h\longrightarrow \h$ is a bounded operator such that
$\{Lf_i\}_{i\in I}$ and $\{L^*f_i\}_{i\in I}$ are frames for $\h$,
then  $\{f_i\}_{i\in I}$ is a frame for $\h$.
\end{prop}
\begin{proof}
Let us define
\[U:{\h}\longrightarrow \h,\quad U(f):=\sum_{i\in
I}\langle f,f_i\rangle f_i.\] Let $S_L$ be the frame operator for
$\{Lf_i\}_{i\in I}$. Then $S_L=LUL^*$ is invertible. So $L$ is
surjective. Similarly, we infer that $L^*$ is surjective. Therefore
$L$ is invertible and so $\{f_i\}_{i\in I}$ is a frame for $\h$ with
the frame operator $L^{-1}S_L(L^*)^{-1}$.
\end{proof}
\begin{prop}
Let $\{f_{i}\}_{i\in I}$ be a Riesz basis for $\h$ with analysis opeartor
$T$, Riesz basis bounds $A\leq B,$ and let $L:\h\longrightarrow \h$ be a bounded opeartor.
Then $\{Lf_{i}\}_{i\in I}$ is a Riesz basis for $\h$ if and only if
$L$ is invertible on $\h$. Moreover in this case the analysis operator
for $\{Lf_{i}\}_{i\in I}$ is $T_{L}=TL^{*}$ and
the new Riesz basis bounds are $\parallel L^{-1}\parallel^{-2}A,$ $\parallel L\parallel^{2}B.$
\end{prop}
\begin{proof}
Since the analysis opeartor for
$\{Lf_{i}\}_{i\in I}$ is $T_{L}=TL^{*},$ $L$ is invertible
if and only if $\{Lf_{i}\}_{i\in I}$ is a Riesz basis for $\h.$
\end{proof}
\begin{cor}
If $\{f_{i}\}_{i\in I}$ is a Riesz basis for $\h$ and $L:\h\longrightarrow\h$ is
a bounded operator, then $\{f_{i}+Lf_{i}\}_{i\in I}$ is a Riesz basis for $\h$ if and only
if $I+L$ is invertible on $\h$. In this case the synthesis operator for new frame is
$T_{I+L}=T(I+L^{*})$ and the new Riesz basis bounds are $\parallel(I+L)^{-1}\parallel^{-2}A$,
$\parallel I+L\parallel^{2}B.$
\end{cor}
\begin{cor}
Let $\{f_{i}\}_{i\in I}$ be a Riesz basis for $\h$ with frame operator $S$
and $\{g_{i}\}_{i\in I}$ be its alternative dual frame. Suppose that
$-1\notin\sigma(S^{-a+b-1}).$
Then $\{S^{a}f_{i}+S^{b}g_{i}\}_{i\in I}$ is a Riesz basis for $\h$ for all real
numbers $a, b$.
\end{cor}
Here, we also show that the equivalence of part (1) and (2) in  Proposition 3.1 of \cite{sum}, is not true in general.
Indeed, if $T_{1}L_{1}^{*}+T_{2}L_{2}^{*}$ is an invertible
operator, then $\{L_{1}f_{i}+L_{2}g_{i}\}_{i\in I}$
is a frame for $\h$ but the inverse is not true.
\begin{exa}
Let $\{e_n\}_{n=1}^{\infty}$ be an orthonormal basis for
$\h$ and $T$ be the analysis operator of $\{e_{n}\}_{n=1}^{\infty}.$
Define a shift operator $L$ on $\h$ as in
Example \ref{el}.  Letting $L_{1}=L_{2}=L$
and $f_{n}=g_{n}=e_{n}$ for each $n\in\mathbb{N},$ in Proposition 3.1 of \cite{sum}, we see
that $\{2L(e_{n})\}_{n=1}^{\infty}$ is a frame for $\h$ but $2TL^{*}$ is not a surjective
operator. If $TL^{*}$ is a surjective operator, then for $\delta_{1}\in l^{2}(\mathbb{N})$,
there exists $h\in H$ such that $TL^{*}(h)=\delta_{1}$ and so $\langle L(e_{1}), h\rangle=1,$ which is
a contradiction.
\end{exa}
\begin{prop}
Let $\{f_{i}\}_{i\in I}$ and $\{g_{i}\}_{i\in I}$ be Bessel sequences in $\h$ with analysis
operators $T_{1}$, $T_{2}$ and frame operators $S_{1}$, $S_{2}$, respectively.
Also let $L_{1}, L_{2}:\h\longrightarrow \h$. Then the following are equivalent:
\\(1) $\{L_{1}f_{i}+L_{2}g_{i}\}_{i\in I}$ is a Riesz basis for $\h$.
\\(2) $T_{1}L_{1}^{*}+T_{2}L_{2}^{*}$ is an invertible operator on $\h$.
\end{prop}
\begin{proof}
$(1)\Leftrightarrow(2)$ $\{L_{1}f_{i}+L_{2}g_{i}\}_{i\in I}$ is a Riesz basis for
$\h$ if and only if its analysis operator $T$ is invertible on $\h$ where
\begin{eqnarray*}
Tf&=&\{\langle f, L_{1}f_{i}+L_{2}g_{i}\rangle\}_{i\in I}\\&=&
\{\langle L_{1}^{*}f, f_{i}\rangle+\langle L_{2}^{*}f, g_{i}\rangle\}_{i\in I}\\&=&
T_{1}L_{1}^{*}f+T_{2}L_{2}^{*}f.
\end{eqnarray*}
\end{proof}
\section{Applications to Gabor frames}
For $x,y\in\Bbb{R}$ we consider the operators $E_x$ and $T_y$ on
$L^2(\Bbb{R})$ defined by $(E_xf)(t)=e^{2\pi ixt}f(t)$ and
$(T_yf)(t)=f(t-y)$. It is easy to prove that $E_x$ and $T_y$ are
unitary with $E^*_x=E_{-x}$ and $T^*y=T_{-y}$. A Gabor frame is a
frame for $L^2(\Bbb{R})$ of the form
$\{E_{mb}T_{na}g\}_{m,n\in\Bbb{Z}}$, where $a,b>0$ and $g\in
L^2(\Bbb{R})$ is a fixed function. We use $(g,a,b)$ to denote
$\{E_{mb}T_{na}g\}_{m,n\in\Bbb{Z}}$.
\begin{lem}\label{lem}
Let $x,y\in\Bbb{R}$ and $c\in\Bbb{C}$ with $|c|=1$. Then the
following are equivalent:
\begin{itemize}
  \item [$(i)$] $I+cT_xE_y$ is  a surjective operator on
  $L^2(\Bbb{R})$.
  \item [$(ii)$] $I+cE_yT_x$ is  a surjective operator on
  $L^2(\Bbb{R})$.
  \end{itemize}
\end{lem}
\begin{proof}
Using Proposition 2 of \cite{cas}, we infer that $I+cT_xE_y$ is
surjective if and only if $I+cT_xE_y$ is invertible. So $I+cT_xE_y$
is invertible if and only if $I+\overline{c}T_{-x}E_{-y}$ is
invertible, and $I+\overline{c}T_{-x}E_{-y}$ is invertible if and
only if $I+cE_yT_x$ is invertible.
\end{proof}
\begin{cor}\label{cor}
Let $x,y\in\Bbb{R}$ and $c\in\Bbb{C}$. If $I+cT_xE_y$ is a
surjective operator on $L^2(\Bbb{R})$, then there exists $\delta>0$
such that $\|(I+cT_xE_y)(g)\|\geqslant \delta\|g\|$ for all $g\in
L^2(\Bbb{R})$.
\end{cor}
In the following, we intend to improve Proposition 4.1 of
\cite{sum}.
\begin{thm}\label{thm}
Let $x,y\in\Bbb{R}$ such that $x\neq0$,  $xy\in\Bbb{Z}$ and let
$c\in\Bbb{C}$ with $|c|=1$. Then
$I+cE_yT_x:L^2(\Bbb{R})\longrightarrow L^2(\Bbb{R})$ is  not
surjective.
\end{thm}
\begin{proof}
It is enough we take $x>0$. Let $f:\Bbb{R}\longrightarrow\Bbb{C}$ be
a function defined by
\[f(t):=\sum_{k=1}^n (-1)^kc^ke^{2\pi ikyt}\chi_{[kx, (k+1)x)}(t).\]
By a simple computation, we get
\[\|f\|^2=\int_{\Bbb{R}}|f(t)|^2\,dt=nx,\quad \|(I+cE_yT_x)f\|^2=2x.\]
Therefore $f\in L^2(\Bbb{R})$ and Corollary \ref{cor} implies that
$I+cE_yT_x$ is not surjective.
\end{proof}
\begin{cor}\label{cor2}
Let $x,y\in\Bbb{R}$ such that $x\neq0$,  $xy\in\Bbb{Z}$ and let
$c\in\Bbb{C}$ with $|c|=1$. If $(g,a,b)$ is a Gabor frame, then
$(g+cE_yT_xg,a,b)$ is not a Gabor frame.
\end{cor}
\begin{proof}
There exists $d\in\Bbb{C}$ with $|d|=1$ such that
$E_{mb}T_{na}(g+cE_yT_xg)=(I+dT_xE_y)(E_{mb}T_{na}g)$. If
$(g+cE_yT_xg,a,b)$ is  a Gabor frame, then $I+dT_xE_y$ is surjective
(invertible) on $L^2(\Bbb{R})$ by Proposition \ref{pro}. So
$I+dE_yT_x$ is surjective by Lemma \ref{lem}. Using Theorem
\ref{thm}, we get a contradiction.
\end{proof}

\end{document}